\documentclass[a4paper, 12pt]{amsart}

\usepackage{amssymb, color}

\theoremstyle{plain}
  \newtheorem{thm}{Theorem}
  \newtheorem{prop}[thm]{Proposition}
  \newtheorem{lem}[thm]{Lemma}
  \newtheorem{cor}[thm]{Corollary}

  \newtheorem{defi}{Definition}

  \newtheorem{thqt}{Theorem}

\theoremstyle{definition}
  \newtheorem{rem}[thm]{Remark}

\numberwithin{equation}{section}

\author[S. Wada]{Shuhei Wada}
\address{Department of Information and Computer Engineering,\\ 
Kisarazu National College of Technology,\\
2-11-1 Kiyomidai-Higashi, Kisarazu,\\ 
Chiba 292-0041, Japan
} 
\email{wada@j.kisarazu.ac.jp}

\author[T. Yamazaki]{Takeaki Yamazaki}
\address{%
  Department of Electrical, Electronic and Computer Engineering\\
  Toyo University\\
  Kawagoe 350-8585, Japan%
} \email{t-yamazaki@toyo.jp}



\title{Equivalence relations among \\
some inequalities on
operator means}

\begin{document}
\maketitle

\begin{abstract}
We will consider about some inequalities on operator means for
more than three operators, for instance, 
ALM and BMP geometric means  will be considered.
Moreover,  
log-Euclidean and logarithmic means
for several operators will be treated.
\end{abstract}

\section{Introduction}
Let $\mathcal{H}$ be a complex Hilbert space, and
$B(\mathcal{H})$ be the algebra of all 
bounded linear operators on $\mathcal{H}$.
An operator $A$ is said to be positive semi-definite
(resp. positive definite)
if and only if $\langle Ax,x\rangle \geq 0$ for 
all $x\in \mathcal{H}$ 
(resp. $\langle Ax,x\rangle > 0$ for 
all non-zero $x\in \mathcal{H}$). We denote positive semi-definite operator 
$A\in B(\mathcal{H})$ by $A\geq 0$.
Let $B(\mathcal{H})_{+}$ and $B(\mathcal{H})_{sa}$ be 
the sets of all positive definite and self-adjoint operators,
respectively.
We can consider the order among $B(\mathcal{H})_{sa}$, i.e.,
for $A, B\in B(\mathcal{H})_{sa}$,
$$ A\leq B \quad \text{if and only if} \quad 0\leq B-A. $$
A real valued function $f$ on an interval $J\subset \mathbb{R}$ 
is called an {\it operator monotone function} if and only if 
$$ A\leq B \quad \mbox{implies}\quad
f(A)\leq f(B) $$
for all $A, B\in B(\mathcal{H})_{sa}$ whose 
spectral are contained in $J$.

For two positive definite operators, the operator 
mean is important in the operator theory.

\begin{defi}[Operator mean {\cite{KA1980}}]
A binary operation $\sigma: B(\mathcal{H})_{+}^{2} \to B(\mathcal{H})_{+}$
is called an {\it operator mean} if and only if the 
following conditions are satisfied.
\begin{enumerate}
\item If $A\leq C$ and $B\leq D$, then $A\sigma B\leq C\sigma D$,
\item $X^{*} (A\sigma B)X\leq (X^{*}AX)\sigma (X^{*}BX)$ for 
$X\in B(\mathcal{H})$,
\item $A_{n}\sigma B_{n}\downarrow A\sigma B$ when 
$A_{n}\downarrow A$ and $B_{n}\downarrow B$ in the strong 
operator topology,
\item $I\sigma I=I$, where $I$ means the identity operator on $\mathcal{H}$.
\end{enumerate}
\end{defi}

We notice that operator means can be defined for 
positive semi-definite operators by (3) in Definition 1. 
Kubo-Ando \cite{KA1980} have shown the following important result:

\begin{thqt}[\cite{KA1980}]\label{thm:Kubo-Ando}
For each operator mean $\sigma$, there exists the 
unique operator monotone function $f: (0,\infty) \longrightarrow (0,\infty)$
such that $f(1)=1$ and 
$$ f(t)I=I\sigma (tI) \quad \text{for all $t\in (0,\infty)$.} $$
Moreover for $A\in B(\mathcal{H})_{+}$ and $B\geq 0$, the formula
$$ A\sigma B=A^{\frac{1}{2}}f(A^{\frac{-1}{2}}BA^{\frac{-1}{2}})A^{\frac{1}{2}}$$
holds, where the right hand side is defined via the analytic functional 
calculus. An operator monotone function $f$ is called the
representing function of $\sigma$.
\end{thqt}

Typical examples of operator means are weighted  
harmonic, geometric and arithmetic means  
denoted by
$!_{w}$, $\sharp_{w}$ and $\nabla_{w}$ for $w\in [0,1]$, respectively.
Their representing functions are 
$[ (1-w)+wt^{-1}]^{-1}$, $t^{w}$ and 
$1-w+wt$, respectively. In fact, we can define 
$A!_{w} B=[(1-w)A^{-1}+wB^{-1}]^{-1}$, 
$A\sharp_{w}B=A^{\frac{1}{2}}(A^{\frac{-1}{2}}BA^{\frac{-1}{2}})^{w}A^{\frac{1}{2}}$ and
$A\nabla_{w}B=(1-w)A+wB$.

Extending Kubo-Ando theory to the theory for three or more operators 
was a long standing problem, in particular, 
we did not have any nice definition of geometric mean for three operators.
Recently, Ando-Li-Mathias have given a 
nice definition of geometric mean for $n$-tuples of
positive definite matrices in \cite{ALM2004}. 
Then many authors 
study about operator means for $n$-tuples of 
positive definite operators, and now we have three 
definitions of geometric means which are called 
ALM, BMP and the Karcher means.
Moreover, we have an extension of the Karcher mean which is called  
the power mean. 

M. Uchiyama and one of the authors have 
obtained equivalence relations between inequalities for
the power and arithmetic means as extensions of
a converse of Loewner-Heinz inequality \cite{UY2014}.

In this paper, we shall investigate the previous research \cite{UY2014} to
other operator means for $n$-tuples of operators.
In fact, we shall treat ALM and BMP means, 
moreover we shall discuss about 
some types of logarithmic means 
of several operators.
This paper is organized as follows.
In Section 2, we will introduce some definitions and notations which will 
be used in this paper. 
Then we shall consider about weighted operator means
in the view point of their representing functions in Section 3.
In Section 4, we shall consider about generalizations
of the results by M. Uchiyama and one of the authors \cite{UY2014}.
Especially, we shall consider about the log-Euclidean mean which is a kind of geometric mean for 
$n$-tuples of positive definite operators. 
In the last section, we shall introduce some properties of the 
$M$-logarithmic mean which is generated from an arbitrary operator mean via an integration.

\section{Primarily}
Let $OM$ be the set of all operator monotone functions
on $(0,\infty)$,
and let $OM_{1}=\{ f\in OM :\ f(1)=1\}$.
For $f\in OM_{1}$, there exists an operator 
mean $\sigma_{f}$ such that
$$ A\sigma_{f}B=
A^{\frac{1}{2}} f(A^{\frac{-1}{2}}BA^{\frac{-1}{2}}) 
A^{\frac{1}{2}} $$
for $A, B\in B(\mathcal{H})_{+}$.
It is well known that for $w\in [0,1]$, if 
$$ A !_{w} B\leq A\sigma_{f} B\leq A\nabla_{w} B $$
holds for all $A, B\in B(\mathcal{H})_{+}$, then 
$$ \left[(1-w)+wt^{-1}\right]^{-1} 
\leq f(t) \leq (1-w)+wt $$
holds for all $t>0$.

Let $A,B\in B(\mathcal{H})_{+}$. The Thompson metric $d(A,B)$ is 
defined by 
$$ d(A,B)=\max\{ \log M(A/B), \log M(B/A)\}, $$
where $M(A/B)=\inf \{\alpha>0\ |\ B\leq \alpha A\}$. 
It is known that a cone of positive definite operators is 
a complete metric space for the Thompson metric.
In what follows, we will consider about ``limit'' of operator sequences 
or ``continuous'' of operator valued functions in the 
Thompson metric without any explanation.

For $n$-tuples of positive definite operators, the ALM and 
BMP (geometric) means are defined as follows.

\begin{thqt}[ALM mean \cite{ALM2004}]\label{thqt1}
For $\mathbb{A}=(A_{1}, A_{2})\in B(\mathcal{H})_{+}^{2}$, 
the ALM (geometric) mean $\mathfrak{G}_{ALM}(\mathbb{A})$
of $\mathbb{A}$ is defined by 
$\mathfrak{G}_{ALM}(\mathbb{A})=A_{1}\sharp_{1/2} A_{2}$.
 Assume that the ALM (geometric) mean
$\mathfrak{G}_{ALM}(\cdot)$ on $B(\mathcal{H})_{+}^{n-1}$ is defined. Let $\mathbb{A}=(A_{1}, \dots, A_{n})\in B(\mathcal{H})_{+}^{n}$ 
and $\{A_i^{(r)}\}_{r=0}^{\infty}$ $(i=1,...,n)$
be the sequences of  positive definite operators defined by
\begin{align*}
A_i^{(0)}=A_i\quad\mbox{and}\quad
A_i^{(r+1)}=\mathfrak{G}_{ALM}\left((A_j^{(r)})_{j\neq i}\right),
\end{align*}
where $(A_j^{(r)})_{j\neq i}=(A_{1}^{(r)},...,A_{i-1}^{(r)},A_{i+1}^{(r)},...,A_{n}^{(r)})$.
Then there exists $\lim_{r\rightarrow\infty}A_i^{(r)}$ 
$(i=1,...,n)$ 
and it does not depend on $i$. 
The ALM (geometric) mean $\mathfrak{G}_{ALM}(\mathbb{A})$ for $n$-tuples of 
positive definite operators $\mathbb{A}\in 
B(\mathcal{H})_{+}^{n}$
is defined by $\lim_{r\rightarrow\infty}A_i^{(r)}$.
\end{thqt}

A vector $\omega=(w_{1},...,w_{n})\in (0,1)^{n}$ is said to be a {\it probability vector} if 
and only if $\sum_{k}w_{k}=1$. Let $\Delta_{n}$ be the set of all 
probability vectors in $(0,1)^{n}$.

\begin{thqt}[BMP mean \cite{BMP2010, IN2009, LLY2011}]\label{thqt2}
For $\mathbb{A}=(A_{1}, A_{2})\in B(\mathcal{H})_{+}^{2}$ and 
$\omega=(1-w, w) \in \Delta_{2}$, 
the BMP (geometric) mean $\mathfrak{G}_{BMP}(\omega; \mathbb{A})$
of $\mathbb{A}$ is defined by 
$\mathfrak{G}_{BMP}(\omega; \mathbb{A})=A_{1}\sharp_{w} A_{2}$.
 Assume that the BMP (geometric) mean
$\mathfrak{G}_{BMP}(\cdot; \cdot)$ on 
$\Delta_{n-1}\times B(\mathcal{H})_{+}^{n-1}$ 
is defined. Let $\mathbb{A}=(A_{1}, \dots, A_{n})\in B(\mathcal{H})_{+}^{n}$ 
and $\omega=(w_{1},...,w_{n})\in \Delta_{n}$.
Define the sequences of  positive definite operators 
$\{A_i^{(r)}\}_{r=0}^{\infty}$ $(i=1,...,n)$ by
\begin{align*}
A_i^{(0)}=A_i\quad\mbox{and}\quad
A_i^{(r+1)}=\mathfrak{G}_{BMP}
\left(\hat{\omega}_{\neq i}; (A_j^{(r)})_{j\neq i}\right)\sharp_{w_{i}}A_{i}^{(r)},
\end{align*}
where $\hat{\omega}_{\neq i}=\frac{1}{\sum_{j\neq i}w_{j}}(w_{j})_{j\neq i}$.
Then there exists $\lim_{r\rightarrow\infty}A_i^{(r)}$ 
$(i=1,...,n)$ and it does not depend on $i$. 
The BMP (geometric) mean $\mathfrak{G}_{BMP}(\omega; \mathbb{A})$
for $n$-tuples of positive definite operators 
$\mathbb{A}\in B(\mathcal{H})_{+}^{n}$
is defined by $\lim_{r\rightarrow\infty}A_i^{(r)}$.
\end{thqt}

We remark that it is not known any weighted ALM mean.
Let $\mathbb{A}=(A_{1},...,A_{n}), \mathbb{B}=(B_{1},...,B_{n})
\in B(\mathcal{H})_{+}^{n}$ and $\omega=(w_{1},...,w_{n})\in \Delta_{n}$.
Here  we denote
the above geometric means 
of $\mathbb{A}$ for the 
weight $\omega$ by $\frak{G}(\omega; \mathbb{A})$,
and they have at least 10 basic properties 
\cite{ALM2004, BMP2010, IN2009, LLY2011}
as follows
(in the ALM mean case, we consider just only $\omega=(\frac{1}{n},...,\frac{1}{n})$ case).

\begin{itemize}
\item[(P1)] If $A_{1},...,A_{n}$ commute with each other, 
then 
$$ \frak{G}(\omega; \mathbb{A})=
\prod_{k=1}^{n}A_{k}^{w_{k}}.$$
\item[(P2)] For positive numbers $a_{1},...,a_{n}$,
$$\frak{G}(\omega;a_{1}A_{1},...,a_{n}A_{n})
=\frak{G}(\omega;a_{1},...,a_{n})
\frak{G}(\omega;\mathbb{A})=
\left(\prod_{k=1}^{n} a_{k}^{w_{k}}\right)\frak{G}(\omega;\mathbb{A}).$$
\item[(P3)] For any permutation $\sigma$
on $\{1,2,...,n\}$,
$$\frak{G}(w_{\sigma(1)},...,w_{\sigma(n)};
A_{\sigma(1)},...,A_{\sigma(n)})=
\frak{G}(\omega; \mathbb{A}).$$
\item[(P4)] If $A_{i}\leq B_{i}$ for $i=1,...,n$, then
$ \frak{G}(\omega; \mathbb{A})
\leq \frak{G}(\omega; \mathbb{B}).$
\item[(P5)]  $\frak{G}(\omega;\cdot)$ is continuous
on each operators. Especially,
$$ d(\frak{G}(\omega;\mathbb{A}), \frak{G}(\omega;\mathbb{B}))
\leq \sum_{i=1}^{n}w_{i}d(A_{i}, B_{i}). $$
\item[(P6)] For each $t\in [0,1]$,
$
 (1-t)\frak{G}(\omega; \mathbb{A})+
t\frak{G}(\omega; \mathbb{B}) 
\leq 
\frak{G}(\omega; (1-t)\mathbb{A}+t\mathbb{B}).
$
\item[(P7)] For any invertible $X\in B(\mathcal{H})$, 
$ \frak{G}(\omega; X^{*}A_{1}X,...,X^{*}A_{n}X)=
X^{*}\frak{G}(\omega; \mathbb{A})X. $
\item[(P8)] 
$\frak{G}(\omega; \mathbb{A}^{-1})^{-1}=
\frak{G}(\omega; \mathbb{A}), $
where $\mathbb{A}^{-1}=(A_{1}^{-1},..., A_{n}^{-1})$.
\item[(P9)] If every $A_{i}$ is a positive definite matrix, then
$\det \frak{G}(\omega; \mathbb{A})=
\prod_{i=1}^{n} \det A_{i}^{w_{i}}. $
\item[(P10)]
$$ \left[ \sum_{i=1}^{n}w_{i}A_{i}^{-1}\right]^{-1}
\leq \frak{G}(\omega; \mathbb{A})
\leq 
\sum_{i=1}^{n}w_{i}A_{i}. $$
\end{itemize}

\section{Operator means of two variables}

In this section, we shall consider the weighted operator means 
in the view point of their weight.

\begin{thm}\label{prop1}
Let $\Phi, f\in OM_{1}$ be non-constant, and let $\sigma$ be 
an operator mean whose representing 
 function is $\Phi$. If $\Phi'(1)=w\in (0,1)$, then  
for $A, B\in B(\mathcal{H})_{sa}$, they are mutually equivalent:
\begin{enumerate}
\item $(1-w)A\leq wB$,
\item $f(\lambda A+I)\sigma f(-\lambda B+I)\leq I$ 
holds for all sufficiently small $\lambda \geq  0$.
\end{enumerate}
\end{thm}

Theorem \ref{prop1} is an extension of the following Theorem \ref{th:2} 
in \cite{UY2014} by Lemma \ref{lem2} introduced in the 
below. 
It was shown as a converse of Loewner-Heinz inequality.

\begin{thqt}[\cite{UY2014}]\label{th:2} Let $f(t)\in OM_{1}$ be non-constant,  and let 
 $A,B\in B(\mathcal{H})_{sa}$. Let $\sigma$ be an operator mean satisfying $! \leq_{1/2} \sigma\leq \nabla_{1/2}$. 
 Then  $A\leq B$ if and only if 
 $f(\lambda A+I) \sigma f(-\lambda B+I) \leq I $  for all sufficiently small $\lambda\geq 0$.  
\end{thqt}

To prove Theorem \ref{prop1}, we need the following lemma.

\begin{lem}\label{lem2}
Let $\Phi\in OM_{1}$.
Then for each $w\in (0,1)$, they are mutually equivalent:
\begin{enumerate}
\item $\Phi'(1)=w$, 
\item $[ (1-w)+wt^{-1}]^{-1}\leq \Phi(t)\leq (1-w)+wt $ for all $t\in (0,\infty)$.
\end{enumerate}
\end{lem}

\begin{proof}
Proof of  (1) $\Longrightarrow$ (2) has been given in \cite[Lemma 2.2]{Ppreprint2014}. 
But we shall introduce its proof for 
the reader's convenience.
Since every operator monotone function is operator concave, 
$\Phi$ is a concave function. We have
$$ \Phi(t)\leq \Phi(1)+\Phi'(1)(t-1)=(1-w)+wt. $$
On the other hand, $\frac{t}{\Phi(t)}$ is also an operator monotone function, 
and 
$$ \frac{d}{dt}\frac{t}{\Phi(t)} \bigg|_{t=1}=1-w. $$
Then by the same argument as above, we have
$$ \frac{t}{\Phi(t)}\leq w+(1-w)t, $$
that is, 
$$ [ (1-w)+wt^{-1}]^{-1}\leq \Phi(t). $$

Conversely, we shall prove (2) $\Longrightarrow$ (1).
Since the tangent line of $ f(t)=[ (1-w)+wt^{-1}]^{-1}$ at $t=1$ is 
$y=(1-w)+wt$, $\Phi(t)$ has the same tangent line of  
$ [ (1-w)+wt^{-1}]^{-1}$ at $t=1$. Therefore $\Phi'(1)=w$.
\end{proof}

Before proving Theorem \ref{prop1}, we introduce the following formulas. 
For any differential function $f$ on $1$ and $w\in (0,1)$,
the following hold in the norm topology.
\begin{align}
& \lim_{\lambda\to 0}f(\lambda A+I)^{\frac{1}{\lambda}}=
e^{f'(1)A} \text{ for $A\in B(\mathcal{H})_{sa}$},
\label{eq: limit-f(1)}\\
& \lim_{p\to 0} \left[(1-w)A^{p}+wB^{p}\right]^{\frac{1}{p}}=
\exp\left[ (1-w)\log A+w\log B\right]
\text{ for $A,B\in B(\mathcal{H})_{+}$},
\label{eq:limit-exp}
\end{align}
where \eqref{eq: limit-f(1)} can be obtained by 
$\lim_{\lambda\to 0}f(\lambda a+1)^{\frac{1}{\lambda}}=e^{f'(1)a}$ for
$a\in \mathbb{R}$, and \eqref{eq:limit-exp} is 
introduced in \cite[(2.4)]{NC1988}, for example.

\begin{proof}[Proof of Theorem \ref{prop1}.]
By Lemma \ref{lem2}, $\Phi'(1)=w$ is equivalent to
\begin{equation}
 [(1-w)+wt^{-1}]^{-1}\leq \Phi(t) \leq (1-w)+wt
\quad\mbox{for all } t>0. 
\label{eq:Harmonic-Arithmetic means}
\end{equation}

We shall prove (1) $\Longrightarrow$ (2).
If $(1-w)A\leq wB$, then it is equivalent to 
$(1-w)(\lambda A+I)+w(-\lambda B+I)\leq I$ for all $\lambda \geq 0$.
Since $f$ is an operator concave function with $f(1)=1$, we have
\begin{align*}
I=f(I) & \geq f\left( (1-w)(\lambda A+I)+w(-\lambda B+I)\right) \\
& \geq (1-w) f(\lambda A+I)+wf(-\lambda B+I) \\
& \geq  f(\lambda A+I)\sigma f(-\lambda B+I),
\end{align*}
where the last inequality holds by \eqref{eq:Harmonic-Arithmetic means}.

Conversely, assume that 
$f(\lambda A+I)\sigma f(-\lambda B+I)\leq I$ for all sufficiently small 
$\lambda \geq 0$. By \eqref{eq:Harmonic-Arithmetic means}, we have
\begin{align*}
I & \geq f(\lambda A+I)\sigma f(-\lambda B+I) \\
& \geq 
\left[ (1-w)f(\lambda A+I)^{-1}+wf(-\lambda B +I)^{-1}\right]^{-1} \\
& \geq 
\left[ (1-w)f(\lambda A+I)^{\frac{-p}{\lambda}}+wf(-\lambda B+I)^{\frac{-p}{\lambda}}\right]^{\frac{-\lambda}{p}}
\end{align*}
for all $0<\lambda \leq p$,
where the last inequality follows from the operator concavity of 
$t^{\alpha}$ for $\alpha\in [0,1]$.
Then we have
$$ 
\left[ (1-w)f(\lambda A+I)^{\frac{-p}{\lambda}}+wf(-\lambda B+I)^{\frac{-p}{\lambda}}\right]^{\frac{-1}{p}}\leq I. $$
By letting $\lambda \to 0$ and \eqref{eq: limit-f(1)}, we have
$$ 
\left[ (1-w)e^{-pf'(1)A}+we^{pf'(1)B}\right]^{\frac{-1}{p}}\leq I, $$
and $p\to 0$, we have
$$ \exp\left( -(1-w)f'(1)A+wf'(1)B\right) \geq I $$
by \eqref{eq:limit-exp}. It is equivalent to $(1-w)A\leq wB$.
\end{proof}

A kind of a converse of Theorem \ref{prop1} 
can be considered as follows.

\begin{prop}\label{prop-converse}
Let $\Phi, f\in OM_{1}$ be non-constant, and let $\sigma$ be 
an operator mean whose representing 
 function is $\Phi$.   
For $A, B\in B(\mathcal{H})_{sa}$ and $w\in (0,1)$, 
if $f(\lambda A+I)\sigma f(-\lambda B+I)\leq I$ 
holds for all sufficiently small $\lambda \geq  0$
whenever $(1-w)A\leq wB$.
Then $\Phi'(1)=w$.
\end{prop}

\begin{proof}
We may assume $f'(1)>0$. 
Let $A=wtI$ and $B=(1-w)tI$ for a real number $t$. 
Then we have $(1-w)A\leq wB$. By the assumption, we have 
$f(\lambda tw+1)\sigma f(-\lambda t(1-w)+1)\leq 1 $ holds for 
all sufficiently small $\lambda \geq 0$. It is equivalent to
$$ \Phi \left(\frac{f(-\lambda t(1-w)+1)}{f(\lambda tw+1)}\right)\leq 
\frac{1}{f(\lambda tw+1)}. $$
For each $\lambda >0$, we have
$$ \frac{ \Phi \left(\frac{f(-\lambda (1-w)t+1)}{f(\lambda tw+1)}\right)-1}
{\lambda}\leq 
\frac{\frac{1}{f(\lambda wt+1)}-1}{\lambda}. $$
Letting $\lambda \to 0$, the right-hand side of the 
above inequality converges to
$$ \frac{\partial}{\partial \lambda} \frac{1}{f(\lambda wt+1)}\biggl|_{\lambda =0}
=
\frac{-wtf'(\lambda wt+1)}{f(\lambda wt+1)^{2}} \biggl|_{\lambda =0}=
-wtf'(1) $$
by the assumption $f(1)=1$.
On the other hand, the left-hand side is 
$$ \frac{\partial}{\partial \lambda} 
\Phi \left(\frac{f(-\lambda (1-w)t+1)}{f(\lambda tw+1)}\right) \biggl|_{\lambda=0}=
-t\Phi'(1)f'(1) $$
by the assumption $\Phi(1)=1$.
Hence we have $t\Phi'(1)\geq wt$ for all real number $t$. Hence we have 
$\Phi'(1)=w$.
\end{proof}

\section{More than three operators case}
Let $\mathbb{A}=(A_{1},...,A_{n})\in B(\mathcal{H})_{+}^{n}$ and $\omega=(w_{1},...,w_{n})\in \Delta_{n}$. Define 
$$\displaystyle \frak{A}(\omega; \mathbb{A})=\sum_{i=1}^{n}w_{i}A_{i}\quad \mbox{and}\quad 
\displaystyle \frak{H}(\omega; \mathbb{A})=
\left(\sum_{i=1}^{n}w_{i} A_{i}^{-1}\right)^{-1}.$$
As an extension of the Karcher mean, the power mean is given by 
Lim-P\'aifia \cite{LP2012}  as follows.
Let $\mathbb{A}=(A_{1},...,A_{n})\in B(\mathcal{H})_{+}^{n}$ and 
$\omega=(w_{1},...,w_{n})\in \Delta_{n}$.
For $t\in (0,1]$, the power mean 
$P_{t}(\omega; \mathbb{A})$ is defined by the unique positive 
definite solution of 
$$ X=\sum_{k=1}^{n}w_{k} X\sharp_{t}A_{k}, $$
and for $t\in[-1,0)$, the power mean 
$P_{t}(\omega; \mathbb{A})$ is defined by 
$P_{t}(\omega; \mathbb{A})=P_{-t}
(\omega; \mathbb{A}^{-1})^{-1}$ (see also \cite{LL2014}).
We remark that $P_{t}(\omega;\mathbb{A})$ converges to 
the Karcher mean $\Lambda(\omega; \mathbb{A})$ as $t\to 0$, strongly. So 
we can consider $P_{0}(\omega; \mathbb{A})$ as $\Lambda(\omega; \mathbb{A})$.
It is known that the Karcher mean also satisfies all properties (P1) -- (P10) in Section 2 (cf. \cite{BH2006, LL2010, LL2014}).
It is easy to see that $P_{1}(\omega; \mathbb{A})=
\frak{A}(\omega; \mathbb{A})$ and  
$P_{-1}(\omega; \mathbb{A})=
\frak{H}(\omega; \mathbb{A})$. Moreover 
$P_{t}(\omega; \mathbb{A})$ is increasing on $t\in [-1,1]$. Hence
the power mean interpolates arithmetic-geometric-harmonic means.
In \cite{UY2014}, we had a generalization of Theorem \ref{th:2} as follows.

\begin{thqt}[\cite{UY2014}]\label{thm: power mean}
Let $T_{1},..., T_{n}$ be Hermitian matrices, and $\omega=(w_{1},...,w_{n})\in \Delta_{n}$.
Let $f \in OM_{1}$ be non-constant.
Then the following assertions are equivalent:
\begin{enumerate}
\item $\displaystyle \sum_{i=1}^{n} w_{i} T_{i}\leq 0$, 
\item $\displaystyle P_{1}(\omega; f(\lambda T_{1}+I),...,f(\lambda T_{n}+I))=\sum_{i=1}^{n}w_{i} f(\lambda T_{i}+I)\leq I$ for all sufficiently small $\lambda\geq 0$, 
\item for each $t\in [-1,1]$, $P_{t} (\omega; f(\lambda T_{1}+I),..., f(\lambda T_{n}+I))\leq I$ for all sufficiently small $\lambda \geq 0$.
\end{enumerate}
\end{thqt}

Here we shall generalize the above result into the following 
Theorem \ref{thm-extension}.

\begin{thm}\label{thm-extension}
Let  $f\in OM_{1}$ be non-constant, 
and let
$\Phi:\ \Delta_{n}\times B(\mathcal{H})_{+}^{n}\times \mathcal{H}\to \mathbb{R}^{+}$ satisfying
\begin{equation}
 \| \frak{H}(\omega; \mathbb{A})\| \leq 
\sup_{\|x\|=1} \Phi(\omega; \mathbb{A}; x) \leq \| \frak{A}(\omega; \mathbb{A})\|
\label{eq:assumption-extension}
\end{equation}
for all $\mathbb{A}\in B(\mathcal{H})_{+}^{n}$ and 
$\omega\in \Delta_{n}$.
Then for $\mathbb{T}=(T_{1},...,T_{n})\in B(\mathcal{H})_{sa}^{n}$
and $\omega=(w_{1},...,w_{n})\in \Delta_{n}$, they are mutually equivalent:
\begin{enumerate}
\item $\displaystyle \sum_{i=1}^{n}w_{i}T_{i}\leq 0$, 
\item $\Phi(\omega; f(\lambda T_{1}+I),..., f(\lambda T_{n}+I); x)\leq 1$ 
for all sufficiently small $\lambda \geq 0$ and all unit vector $x\in \mathcal{H}$.
\end{enumerate}
\end{thm}

In fact, we obtain Theorem \ref{thm: power mean} by putting 
$ \Phi (\omega; \mathbb{A}; x)=\langle P_{t}(\omega; \mathbb{A})x,x\rangle $
in Theorem \ref{thm-extension}.

\begin{proof}[Proof of Theorem \ref{thm-extension}]
First of all, we may assume $f'(1)>0$.
Firstly, we shall prove (1) $\Longrightarrow$ (2).
For each $\lambda > 0$, (1) is equivalent to 
$$ \sum_{i=1}^{n}w_{i}(\lambda T_{i}+I)\leq I. $$
Since operator concavity of $f$ and $f(1)=1$, we have
\begin{align*}
I=f(I) 
& \geq 
f\left( \sum_{i=1}^{n}w_{i}(\lambda T_{i}+I)\right) \\
& \geq 
 \sum_{i=1}^{n}w_{i}f(\lambda T_{i}+I) 
 =
\frak{A}(\omega;  f(\lambda T_{1}+I),...,f(\lambda T_{n}+I)).
\end{align*}
Here by \eqref{eq:assumption-extension}, 
\begin{align*}
1 & \geq \| \frak{A}(\omega;  f(\lambda T_{1}+I),...,f(\lambda T_{n}+I))\| \\
& \geq 
\sup_{\|x\|=1} \Phi(\omega; f(\lambda T_{1}+I),...,f(\lambda T_{n}+I); x),
\end{align*}
we have
$$ 1\geq \Phi(\omega; f(\lambda T_{1}+I),..., f(\lambda T_{n}+I); x) $$
for all unit vector $x\in \mathcal{H}$, i.e., (2).

Conversely, we shall prove (2) $\Longrightarrow$ (1).
By \eqref{eq:assumption-extension}, we have
\begin{align*}
1 & \geq
\sup_{\|x\|=1} \Phi(\omega; f(\lambda T_{1}+I),..., f(\lambda T_{n}+I); x)\\
& \geq 
\| \frak{H}(\omega;  f(\lambda T_{1}+I),...,f(\lambda T_{n}+I))\| .
\end{align*}
Then 
\begin{align*}
I & \geq 
\frak{H}(\omega;  f(\lambda T_{1}+I),...,f(\lambda T_{n}+I))\\
& =
\left[\sum_{i=1}^{n}w_{i}f(\lambda T_{i}+I)^{-1}\right]^{-1}
 \geq 
\left[\sum_{i=1}^{n}w_{i}f(\lambda T_{i}+I)^{\frac{-p}{\lambda}}\right]^{\frac{-\lambda}{p}}
\end{align*}
for all $0<\lambda\leq p$ since $t^{\alpha}$ is operator concave for 
$\alpha\in [0,1]$.
Hence we have
$$ \left[\sum_{i=1}^{n}w_{i}f(\lambda T_{i}+I)^{\frac{-p}{\lambda}}\right]^{\frac{-1}{p}}
\leq I. $$
By letting $\lambda \to 0$ and \eqref{eq: limit-f(1)}, we obtain
$$ \left[\sum_{i=1}^{n}w_{i}e^{-pf'(1)T_{i}}\right]^{\frac{-1}{p}}
\leq I, $$
and $p\to 0$, we have
$ f'(1)\sum_{i=1}^{n}w_{i}T_{i}\leq 0$, that is, (1).
\end{proof}

\begin{cor}\label{Corollary norm}
Let  $f\in OM_{1}$ be non-constant.
Then for $\mathbb{T}=(T_{1},...,T_{n})\in B(\mathcal{H})_{sa}^{n}$ and 
$\omega=(w_{1},...,w_{n})\in \Delta_{n}$,
they are mutually equivalent:
\begin{enumerate}
\item $\displaystyle \sum_{i=1}^{n}w_{i}T_{i}\leq 0$, 
\item $\prod_{i=1}^{n} \| f(\lambda T_{i}+I)^{\frac{1}{2}}x\|^{w_{i}} \leq 1$ 
for all sufficiently small $\lambda > 0$ and all unit vector $x\in \mathcal{H}$,
%
\end{enumerate}
\end{cor}

\begin{proof}
For $\mathbb{A}=(A_{1},...,A_{n})\in B(\mathcal{H})_{+}^{n}$, let 
 $\Phi(\omega; \mathbb{A}; x)=\prod_{i=1}^{n} 
\| A_{i}^{\frac{1}{2}}x\|^{2w_{i}}$. 
We shall only check  
$$ \| \frak{H}(\omega; \mathbb{A})\| \leq \sup_{\|x\|=1}\prod_{i=1}^{n} \| A_{i}^{\frac{1}{2}}x\|^{2w_{i}} \leq 
\| \frak{A}(\omega; \mathbb{A})\| $$
for all $\mathbb{A}=(A_{1},...,A_{n})\in B(\mathcal{H})_{+}^{n}$ and $\omega=(w_{1},...,w_{n})\in \Delta_{n}$.
Firstly, we shall show 
$\sup_{\|x\|=1}\prod_{i=1}^{n} \| A_{i}^{\frac{1}{2}}x\|^{2w_{i}} \leq 
\| \frak{A}(\omega; \mathbb{A})\|$.
\begin{align*}
\prod_{i=1}^{n} \| A_{i}^{\frac{1}{2}}x\|^{2w_{i}} 
 =
\prod_{i=1}^{n} \langle A_{i}x,x\rangle^{w_{i}} 
 \leq 
\sum_{i=1}^{n} w_{i} \langle A_{i}x,x\rangle
 =
\langle \frak{A}(\omega; \mathbb{A})x,x\rangle. 
\end{align*}
Hence, we have $\sup_{\|x\|=1}\prod_{i=1}^{n} \| A_{i}^{\frac{1}{2}}x\|^{2w_{i}} \leq 
\| \frak{A}(\omega; \mathbb{A})\|$.

Next, we shall prove 
$ \| \frak{H}(\omega; \mathbb{A})\| \leq \sup_{\|x\|=1}\prod_{i=1}^{n} \| A_{i}^{\frac{1}{2}}x\|^{2w_{i}}$.
\begin{align*}
\prod_{i=1}^{n} \| A_{i}^{\frac{1}{2}}x\|^{2w_{i}}
& =
\prod_{i=1}^{n} \langle  A_{i}x,x \rangle^{w_{i}} \\
& \geq 
\langle \Lambda(\omega; \mathbb{A}) x,x\rangle 
\quad \text{(by \cite{Y2013})}\\
& \geq 
\langle \frak{H}(\omega; \mathbb{A})x,x\rangle.
\end{align*}
Therefore the proof is completed by Theorem \ref{thm-extension}.
\end{proof}

Corollary \ref{Corollary norm} is an extension of the following result:

\begin{thqt}[\cite{UY2014}]\label{thm: Karcher mean non-weighted}
Let $T_{1},..., T_{n}$ be Hermitian matrices, and 
let $f\in OM_{1}$ be non-constant.
Then the following are equivalent:
\begin{enumerate}
\item $\displaystyle \sum_{i=1}^{n}T_{i}\leq 0$, 
\item $\displaystyle \|x\|^{n} \leq \prod_{i=1}^{n} \|f(\lambda T_{i}+I)^{\frac{-1}{2}} x\|$ for all 
sufficiently small $\lambda\geq 0$ and all  $x\in \mathcal{H}$.
\end{enumerate}
\end{thqt}

From here we shall consider 
another geometric mean for  $n$-tuples of positive definite operators
which is called the log-Euclidean mean $\frak{G}_{E}(\omega; \mathbb{A})$
for $\mathbb{A}=(A_{1},...,A_{n})\in B(\mathcal{H})_{+}^{n}$ and $\omega=(w_{1},...,w_{n})\in \Delta_{n}$. 
It is defined by
$$ \frak{G}_{E}(\omega; \mathbb{A})=\exp\left(\sum_{i=1}^{n}w_{i}\log A_{i}\right). $$
Log-Euclidean mean satisfies some of properties (P1)--(P10) in Section 2.
However, log-Euclidean mean does not satisfy 
important properties (P4) and (P10). 

\begin{cor}\label{cor1}
Let  $f\in OM_{1}$ be non-constant.
For $\mathbb{A}\in B(\mathcal{H})_{+}^{n}$ and  
$\omega\in \Delta_{n}$, 
let $M(\omega; \mathbb{A})$ be 
ALM or weighted BMP or log-Euclidean mean
(in the ALM mean case, $\omega$ should be $\omega=(\frac{1}{n},...,\frac{1}{n})$).
Then for $\mathbb{T}=(T_{1},...,T_{n})\in B(\mathcal{H})_{sa}^{n}$ and 
$\omega=(w_{1},...,w_{n})\in \Delta_{n}$,
the following assertions are equivalent:
\begin{enumerate}
\item $\displaystyle \sum_{i=1}^{n}w_{i}T_{i}\leq 0$,
\item $\displaystyle 
M(\omega; f(\lambda T_{1}+I),...,f(\lambda T_{n}+I))\leq I$
for all sufficiently small $\lambda\geq 0$.
\end{enumerate}
\end{cor}

\begin{proof}
The cases of ALM and BMP means.
Put $\Phi(\omega; \mathbb{A};x)=\langle M(\omega; \mathbb{A})x,x\rangle$.
Then by (P10), $\Phi(\omega; \mathbb{A};x)$ 
satisfies the condition \eqref{eq:assumption-extension}. 
So that we can prove the cases of ALM and BMP means by Theorem \ref{thm-extension}.

By the way, log-Euclidean mean satisfies
\begin{equation}
 \log \frak{H}(\omega; \mathbb{A}) \leq 
\log \frak{G}_{E}(\omega; \mathbb{A}) \leq 
\log \frak{A}(\omega; \mathbb{A})
\label{log-inequality log-Euclidean mean}
\end{equation}
for $\mathbb{A}\in B(\mathcal{H})_{+}^{n}$ and $\omega\in \Delta_{n}$.
In fact, by the operator concavity of $\log t$, we have
\begin{align*}
\log \frak{G}_{E}(\omega; \mathbb{A}) & = 
\log \left[ \exp\left(\sum_{i=1}^{n}w_{i}\log A_{i}\right) \right]\\
& =
\sum_{i=1}^{n}w_{i}\log A_{i} \\
& \leq 
\log \left(\sum_{i=1}^{n}w_{i}A_{i}\right) =\log \frak{A}(\omega; \mathbb{A}).
\end{align*}

On the other hand, we have 
\begin{align*}
\log \frak{H}(\omega; \mathbb{A}) & = 
\log \frak{A}(\omega; \mathbb{A}^{-1})^{-1} \\
& = 
- \log \frak{A}(\omega; \mathbb{A}^{-1}) \\
& \leq
- \log \frak{G}_{E}(\omega; \mathbb{A}^{-1}) \\
& =
\log \frak{G}_{E}(\omega; \mathbb{A}).
\end{align*}
Hence we have \eqref{log-inequality log-Euclidean mean}.
We remark that if $\log A\leq \log B$ for $A,B\in B(\mathcal{H})_{+}$, 
then for each $p>0$, there is a unitary operator $U_{p}$
such that $A^{p}\leq U_{p}^{*}B^{p}U_{p}$ in \cite{F1997}. Hence we have 
$\|A\|\leq \|B\|$. By using this fact to \eqref{log-inequality log-Euclidean mean}, 
we have
$$  \| \frak{H}(\omega; \mathbb{A})\| \leq 
\| \frak{G}_{E}(\omega; \mathbb{A}) \| \leq 
\| \frak{A}(\omega; \mathbb{A})\|. 
$$
Hence we can prove Corollary \ref{cor1} by putting 
$\Phi(\omega; \mathbb{A};x)=\langle \frak{G}_{E}(\omega; \mathbb{A})x,x\rangle$ in 
 Theorem \ref{thm-extension}.
\end{proof}

\section{Logarithmic means}

We shall consider some logarithmic means for $n$-tuples of positive definite operators.
Since the representing function of logarithmic mean is $\frac{t-1}{\log t}$,
logarithmic mean $A\lambda B$ of $A, B\in B(\mathcal{H})_{+}$ 
can be considered as 
$$ A\lambda B =\int_{0}^{1} A\sharp_{t} B dt. $$
So it is quite natural to consider the similar 
type of integrated means as follows.

\begin{defi}[$M$-logarithmic mean]
Let $M:\ \Delta_{n}\times B(\mathcal{H})_{+}^{n}\to B(\mathcal{H})_{+}$.
Then for $\mathbb{A}\in B(\mathcal{H})_{+}^{n}$, 
the $M$-logarithmic mean $\frak{L}(M)(\mathbb{A})$ of 
$\mathbb{A}\in B(\mathcal{H})_{+}^{n}$ is defined by 
$$  \frak{L}(M)(\mathbb{A}):=\int_{\Delta_{n}}
M(\omega; \mathbb{A}) dp(\omega) $$
if there exists, where $dp(\omega)$ means an arbitrary probability measure 
on $\Delta_{n}$.
\end{defi}

In what follows, we consider the case of $dp(\omega)=(n-1)!d\omega$.

\begin{prop}\label{prop:logarithmic mean}
Let $M: \Delta_{n}\times B(\mathcal{H})_{+}^{n}\to B(\mathcal{H})_{+}$ satisfying 
(P3), (P7), (P8) and (P10). Then $M$-logarithmic mean
$$ \frak{L}(M)(\mathbb{A})= (n-1)! \int_{\Delta_{n}} M(\omega; \mathbb{A})d \omega $$
satisfies (P3) and  (P7) if it exists.
Especially, $\frak{L}(M)$ satisfies (P10), i.e.,  
$$ \frak{H}(\mathbb{A})\leq 
\frak{L}(M)(\mathbb{A})\leq \frak{A}(\mathbb{A}).$$
\end{prop}

We remark that $\frak{L}(\frak{A})(\mathbb{A})=\frak{A}(\mathbb{A})$.
As for the preparation, we define some notations.
Let $S$ be the cyclic shift operator on $\mathbb{C}^{n}$ and let $\mathbb{S}$
be also the cyclic shift operator on $B(\mathcal{H})^{n}$; namely,
\begin{align*}
S(w_{1},w_{2},...,w_{n}) & = (w_{2},w_{3},...,w_{n}, w_{1}). \\
\mathbb{S}(A_{1},A_{2},...,A_{n}) & = (A_{2},A_{3},...,A_{n}, A_{1}).
\end{align*}
We claim that if $M$ satisfies (P3), then 
$M(S\omega; \mathbb{A})=M(\omega; \mathbb{S}^{*}\mathbb{A})$.

\begin{proof}[Proof of Proposition \ref{prop:logarithmic mean}.] 
It is clear that $\frak{L}(M)$ satisfies (P3) and (P7). 
The remain is to show (P10).
Let $\mathbb{M}$ be the set of all maps  
$M: \Delta_{n}\times B(\mathcal{H})_{+}^{n}\to B(\mathcal{H})_{+}$. 
It is easy to show that $\frak{L}$ is a linear map 
on $\mathbb{M}$, and 
$\frak{L}(M)(\mathbb{A})\geq 0$ for all $\mathbb{A}\in B(\mathcal{H})^{n}_{+}$ if
$M\in \mathbb{M}$. Hence for  $N_{1}, M, N_{2}\in \mathbb{M}$,
if $N_{1}(\omega; \mathbb{A})\leq M(\omega; \mathbb{A}) \leq 
N_{2}(\omega; \mathbb{A})$ holds for all $\omega\in \Delta_{n}$ and
$\mathbb{A}\in B(\mathcal{H})_{+}^{n}$, then 
$$ \frak{L}(N_{1})(\mathbb{A}) \leq \frak{L}(M)(\mathbb{A}) \leq \frak{L}(N_{2})(\mathbb{A})$$
holds for all $\mathbb{A}\in B(\mathcal{H})^{n}_{+}$.
Since $M(\omega;\mathbb{A})$ satisfies (P10), we have
\begin{align*}
\frak{L}(M)(\mathbb{A}) & =
(n-1)!\int_{\Delta_{n}}M(\omega; \mathbb{A})d \omega \\
& \leq
(n-1)!\int_{\Delta_{n}}\frak{A}(\omega; \mathbb{A})d \omega \\
& =
\frak{L}(\frak{A})(\mathbb{A})=
\frak{A}(\mathbb{A}).
\end{align*}
On the other hand, we have 
\begin{align*}
\frak{H}(\mathbb{A}) 
& = 
\frak{A}(\mathbb{A}^{-1})^{-1} \\
& \leq 
\{\frak{L}(M)(\mathbb{A}^{-1})\}^{-1} \\
& =
\left( (n-1)!\int_{\Delta_{n}} M(\omega; \mathbb{A}^{-1})d \omega\right)^{-1} \\
& =
\left( (n-1) \int_{\Delta_{n}} M(\omega; \mathbb{A})^{-1}d \omega \right)^{-1} 
\quad\mbox{ by (P8)} \\
& \leq 
(n-1)! \int_{\Delta_{n}} M(\omega; \mathbb{A})d \omega 
 =
\frak{L}(M)(\mathbb{A}).
\end{align*}

\end{proof}

\begin{rem}
The above theorem is valid for an arbitrary permutation (shift)-
invariant probability measure $p$.
\end{rem}

\begin{rem}
Let $M:\ \Delta_{n}\times B(\mathcal{H})_{+}^{n}\to B(\mathcal{H})_{+}$ be a map satisfying (P3), (P7), (P8) and (P10).
We put 
$$ M_{0}(\omega; \mathbb{A}):=
M\left( (\frac{1}{n},...,\frac{1}{n}); 
M(\omega; \mathbb{A}), M(S\omega; \mathbb{A}),...,
M(S^{n-1}\omega; \mathbb{A})\right).$$
Then $M_{0}$ satisfies the assumption of Proposition \ref{prop:logarithmic mean}.
So $\frak{L}(M_{0})$ also satisfies (P10).
Moreover, the following inequalities hold 
$$ \frak{H}(\mathbb{A})\leq \frak{L}(M_{0})(\mathbb{A})\leq 
\frak{L}(M)(\mathbb{A})\leq \frak{A}(\mathbb{A}). $$
The second inequality can be shown as follows.
Since $M(\omega; \mathbb{A})$ satisfies (P10), we have
$$M_{0}(\omega; \mathbb{A})\leq \sum_{k=0}^{n-1}\frac{1}{n}M(S^{k}\omega;\mathbb{A}).$$
Then we obtain
\begin{align*}
\frak{L}(M_{0})(\mathbb{A}) & = 
(n-1)!\int_{\Delta_{n}} M_{0}(\omega;\mathbb{A})d\omega \\
& \leq 
(n-1)!\int_{\Delta_{n}} \left\{ \sum_{k=0}^{n-1}\frac{1}{n}M(S^{k}\omega;\mathbb{A})\right\} d\omega \\
& = 
\frac{(n-1)!}{n}  \sum_{k=0}^{n-1}\int_{\Delta_{n}} M(S^{k}\omega;\mathbb{A}) d\omega \\
& = 
\frac{1}{n}  \sum_{k=0}^{n-1}\frak{L}(M)(\mathbb{A})
 = \frak{L}(M)(\mathbb{A}).
\end{align*}
%
\end{rem}

Since the weighted Karcher mean $\Lambda(\omega;\mathbb{A})$ 
is continuous on the probability vector 
in the Thompson metric \cite{LL2014}, so $\frak{L}(\Lambda)(\mathbb{A})$
exists.

\begin{prop}\label{prop4}
$$ \frak{H}(\mathbb{A})\leq \frak{L}(\Lambda)(\mathbb{A})\leq 
\frak{A}(\mathbb{A}). $$
\end{prop}

\begin{proof}
Since the weighted Karcher mean satisfies 
(P1)--(P10) in Section 2 \cite{BH2006, LL2010, LL2014}, 
it is easy by Proposition \ref{prop:logarithmic mean}.
\end{proof}

\begin{cor}\label{cor5}
Logarithmic mean $\frak{L}(\Lambda)(\mathbb{A})$ satisfies the same assertion to 
Corollary \ref{cor1}, too.
\end{cor}

\begin{proof}
We can prove Corollary \ref{cor5} by the same way to the proof of 
Corollary \ref{cor1}.
\end{proof}

\end{document}